\documentclass[a4paper,11pt,reqno]{amsart}

\usepackage{spiros}
\usepackage[pdftex]{graphicx}
\usepackage[all]{xy}

\title{Triangular homotopy equivalences}
\author{Spiros Adams-Florou}

\newtheorem*{thmm*}{Topological squeezing theorem}
\newtheorem*{conj*}{Conjecture}
\newtheorem*{ack*}{Acknowledgement}

\begin{document}

\maketitle

\begin{abstract}
A map $f:X\to Y$ to a simplicial complex $Y$ is called a $Y$-triangular homotopy equivalence if it has a homotopy inverse $g$ and homotopies $h_1:f\circ g\simeq \id_Y$, $h_2:g\circ f\simeq \id_X$ such that for all simplices $\sigma\in Y$, $f|_\sigma:f^{-1}(\sigma) \to \sigma$ is a homotopy equivalence with inverse $g|_\sigma:\sigma \to f^{-1}(\sigma)$ and homotopies $h_1|_\sigma$ and $h_2|_\sigma$. In this paper we prove that for all pairs $X,Y$ of finite-dimensional locally finite simplicial complexes there is an $\ep(X,Y)>0$ such that any $\ep$-controlled homotopy equivalence $f:X\to Y$ for $\ep<\ep(X,Y)$ is homotopic to a $Y$-triangular homotopy equivalence. Conversely, we conjecture that it is possible to `subdivide' a $Y$-triangular homotopy equivalence by finding a homotopic $(Sd\, Y)$-triangular homotopy equivalence, consequently a $Y$-triangular homotopy equivalence would be homotopic to an $\ep$-controlled homotopy equivalence for all $\ep>0$.
\end{abstract}

\section{Introduction}
Let $X,Z$ be topological spaces equipped with control maps $p:X\to Y$, $q:Z\to Y$ to a simplicial complex $Y$. We say that a map $f:(X,p)\to (Z,q)$ is \textit{$Y$-triangular} if it sends the preimage under $p$ of each simplex interior to the preimage under $q$ of the same simplex: \[q(f(p^{-1}(\mathring{\sigma}))) \subset \sigma,\quad \forall \sigma\in Y.\]
Two maps $f,g:(X,p)\to (Z,q)$ are \textit{$Y$-triangular homotopic} if there is a homotopy \[H:f\sim g:(X\times [0,1],p\circ pr_1) \to (Z,q)\] that is \textit{$Y$-triangular}. The composition of $Y$-triangular maps is again $Y$-triangular and being $Y$-triangular homotopic is an equivalence relation. We say that a map $f: (X,p) \to (Z,q)$ is a \textit{$Y$-triangular homotopy equivalence} if there exists a homotopy inverse $g:Z \to X$ and homotopies $h_1:f\circ g\simeq \id_Z$, $h_2:g\circ f\simeq \id_X$ such that all of $f$, $g$, $h_1$ and $h_2$ are $Y$-triangular. Equivalently the restrictions of $f$, $g$, $h_1$ and $h_2$ give homotopy equivalence over each closed simplex of $Y$. 

To avoid a $Y$-triangular homotopy equivalence being too weak a condition we would like the control maps not to lose too much information, i.e. that they are highly connected. To this end we restrict our attention in this paper to simplicial maps between finite-dimensional locally finite simplicial complex. Such a simplicial complex may be equipped with the path metric whose restriction to each $n$-simplex is the subspace metric obtained from the standard embedding of the $n$-simplex into $\R^{n+1}$. Thus, we may take the control space to be the target and the control maps to be $f$ and the identity respectively; we study $Y$-triangular maps of the form \[f:(X,f) \to (Y,\id_Y).\] 

The definition of a $Y$-triangular homotopy equivalence was made in the attempt to find a metric-free condition corresponding to the notion of an $\ep$-controlled homotopy equivalence for all $\ep>0$: a map $f:(X,p) \to (Z,q)$ where the control space $Y$ is equipped with a metric $d$ is called \textit{$\ep$-controlled} if \[ d(p(x),q(f(x))<\ep,\quad \forall x\in X.\] Similarly $f:(X,p)\to (Z,q)$ is an \textit{$\ep$-controlled homotopy equivalence} if there exists a homotopy inverse $g$ and homotopies $h_1:f\circ g\simeq \id_Z$, $h_2:g\circ f\simeq \id_X$ such that all of the maps $f$, $g$, $h_1$ and $h_2$ are $\ep$-controlled. 

The main result of this paper is that a homotopy equivalence $f:(X,f)\to (Y,\id_Y)$ with small enough control is homotopic to a $Y$-triangular homotopy equivalence:
\begin{thmm*}
Let $Y$ be a finite-dimensional locally finite simplicial complex. There exists an $\ep(Y)>0$ such that if there is an $\ep$-controlled homotopy equivalence $f: (X,f)\simeq (Y,\id_Y)$ for some $\ep<\ep(Y)$, then $f$ is homotopic to a $Y$-triangular homotopy equivalence.
\end{thmm*}

We conjecture that it is possible to `subdivide' a $Y$-triangular homotopy equivalence:
\begin{conj*}
A $Y$-triangular homotopy equivalence is $Y$-triangular homotopic to an $(Sd\, Y)$-triangular homotopy equivalence.  
\end{conj*}
If this conjecture holds then by subdividing repeatedly one could find a homotopic $\ep$-controlled homotopy equivalence for $\ep$ as small as you like. The map obtained in the limit would be an $\ep$-controlled homotopy equivalence for all $\ep>0$.

The motivation for this entire approach comes from algebra where an algebraic squeezing result for chain complexes and its converse both hold. This is in the context of geometric categories over a simplicial complex. Let $\A$ be an additive category, $Y$ a finite-dimensional locally finite (henceforth f.d. l.f.) simplicial complex and let $\BB(\A)$ denote the category of bounded chain complexes in $\A$. Recall the definition of the $Y$-controlled category $\A^*(Y)$ of  Ranicki and Weiss \cite{ranickiweiss}: Objects of $\A^*(Y)$ are collections of objects of $\A$, $\{M(\sigma)\,|\, \sigma \in Y \}$, indexed by the simplices of $Y$, written as a direct sum \[\sum_{\sigma\in Y} M(\sigma).\] Morphisms of $\A^*(Y)$ \[f = \{f_{\tau,\sigma}\}: L = \sum_{\sigma\in Y}L(\sigma) \to M = \sum_{\tau\in Y}M(\tau) \]are collections $\{f_{\tau,\sigma}:L(\sigma) \to M(\tau)\,|\,\sigma,\tau \in Y \}$ of morphisms in $\A$ such that $f_{\tau,\sigma}:M(\sigma)\to N(\tau)$ is $0$ unless $\tau\leqslant \sigma$. In particular notice that morphisms of $\A^*(Y)$ 
can be thought of as triangular matrices.

There is an inclusion of categories $\iota:\A^*(Y)\hookrightarrow \CC_{||Y||}(\A)$ where $\CC_{||Y||}(\A)$ is the bounded category of Pedersen and Weibel \cite{kthyhom} and $||Y||$ is the geometric realisation of $Y$. The inclusion is given by associating simplices to their barycentres and observing that a morphism in $\A^*(Y)$ has control at most $\mesh(Y)<\infty$ where \[\mesh(Y):= \sup_{\sigma\in Y}\{\diam(\sigma)\}.\] In \cite{mythesisarxiv} an algebraic subdivision functor 
\[Sd\,: \BB(\A^*(Y))\to \BB(\A^*(Sd\, Y))\]
is constructed together with an assembly functor \[\rr:\BB(\A^*(Sd\, Y)) \to \BB(\A^*(Y))\] such that for all $C\in\BB(\A^*(Y))$, $\rr Sd\, C \simeq C$ in $\BB(\A^*(Y))$. The control of $Sd\, C$ is at most $\mesh(Sd\, Y) \leqslant \dfrac{\dim(Y)}{\dim(Y)+1}\mesh(Y)$, hence the control of $Sd^i\, C$ tends to $0$ as $i\to \infty$. This means that 
chain complexes in $\iota(\BB(\A^*(Y)))$ are boundedly chain equivalent to ones with arbitrarily small control. There is an algebraic squeezing theorem:
\begin{thm*}[\cite{mythesisarxiv}]
Let $Y$ be an f.d. l.f. simplicial complex. Then there exists an $\ep(Y)>0$ such that for all $\ep<\ep(Y)$ there exists an integer $i=i(Y,\ep)$ such that for $C,D \in \BB(\A^*(Y))$ if there is a chain equivalence $\iota(Sd^i\,C) \simeq \iota(Sd^i\, D)$ in $\BB(\CC_{||Y||}(\A))$ with control at most $\ep$, then there is a chain equivalence $C\simeq D$ in $\BB(\A^*(Y))$.
\end{thm*}
Thus, for $C,D\in \BB(\A^*(Y))$, chain equivalences $f:C\to D$ in $\BB(\A^*(Y))$ (i.e. those with triangular matrices) correspond to chain equivalences which are chain homotopic to $\ep$-controlled ones for all $\ep>0$. This statement motivated the idea that $Y$-triangular homotopy equivalences should correspond to homotopy equivalences which are homotopic to $\ep$-controlled ones for all $\ep>0$.

In section \ref{section1} we present some necessary definitions and prove a technical lemma. In section \ref{section2} we prove the main theorem and conclude with our conjecture.

\begin{ack*}
This work is partially supported by Prof. Michael Weiss' Humboldt Professorship.
\end{ack*}



\section{Preliminaries}\label{section1}
In this paper only locally finite finite-dimensional simplicial complexes will be considered. Such a space $X$ shall be given a metric $d_X$, called the \textit{standard metric}, as follows. First define the \textit{standard $n$-simplex} $\Delta^n$ in $\R^{n+1}$ as the join of the points $e_0 = (1,0,\ldots, 0)$, \ldots, $e_n = (0,\ldots,0,1)$. $\Delta^n$ is given the subspace metric $d_{\Delta^n}$ of the standard $\ell_2$-metric on $\R^{n+1}$. The locally finite finite-dimensional simplicial complex $X$ is then given the path metric whose restriction to each $n$-simplex is $d_{\Delta^n}$. Distances between points in different connected components are thus $\infty$. See $\S 4$ of \cite{bartelssqueezing} or Definition $3.1$ of \cite{HR95} for more details.

Let $p:X\to Y$ be a simplicial control map between locally-finite simplicial complexes equipped with standard metrics and let $\sigma$ be a simplex in $X$. The \textit{diameter of $\sigma$ measured in $Y$} is \[\mathrm{diam}(\sigma):= \sup_{x,y\in\sigma}{d_Y(p(x),p(y))}.\] The \textit{radius of $\sigma$ measured in $Y$} is \[\mathrm{rad}(\sigma) := \inf_{x\in\partial\sigma}d_Y(p(\widehat{\sigma}),p(x))\] where $\widehat{\sigma}$ denotes the barycentre of $\sigma$. The \textit{mesh of $X$ measured in $Y$} is \[\mesh(X):= \sup_{\sigma\in X}\{\mathrm{diam}(\sigma)\}.\] The \textit{comesh of $X$ measured in $Y$} is \[\comesh(X):= \inf_{\sigma\in X, |\sigma|\neq 0}\{\mathrm{rad}(\sigma)\}.\] We take the convention that subdivision will make simplices smaller by measuring $Y^\prime$ in the same control space as $Y$.

Using the standard metric on $Y$ and $\id_Y:Y\to Y$ as the control map $\diam (\sigma) = \sqrt{2}$ and $\rad (\sigma) = \dfrac{1}{\sqrt{|\sigma|(|\sigma|+1)}},$ for all $\sigma\in Y$, so consequently $\mesh (Y) = \sqrt{2}$ and if $Y$ is $n$-dimensional $\comesh (Y) = \dfrac{1}{\sqrt{n(n+1)}}$.

For any simplex $\sigma\in Y$ define the \textit{(closed) dual cell} , $D(\sigma, Y)$, by \[D(\sigma,Y):= \{ \widehat{\sigma}_0\ldots\widehat{\sigma}_k\in Sd\, Y\,|\,\sigma\leqslant\sigma_0<\ldots<\sigma_k\in Y\}.\] 

In this paper we consider a simplicial map $r:Sd^j\,X \to X$ to be a \textit{simplicial approximation to the identity} if and only if $r\simeq\id_X$ \textit{and} for all $\sigma \in X$, $r(Sd^j\,\sigma)\subset \sigma$.

It is not true in general that preimages commute with $\ep$-neighbourhoods, however for a surjective simplicial map $f:X\to Y$ the following does hold: 

\begin{lem}\label{annoyinglemma}
Let $f:\sigma \to \tau$ be a surjective simplicial map between simplices with subspace metrics induced by p.l. embeddings into Euclidean space. Let $\rho$ be a proper subsimplex of $\tau$. Then \[f^{-1}(N_{k\ep}(\rho)) \subset N_\ep(f^{-1}(\rho)) \subset f^{-1}(N_{K\ep}(\rho)),\]for $k = \dfrac{\diam(\tau)}{2\,\rad(\sigma)}$, $K= \dfrac{2\,\rad(\tau)}{\diam(\sigma)}$.
\end{lem}
\begin{proof}
Let $\rho < \tau$, then there exists a unique simplex $\rho^\prime<\tau$ such that $\tau = \rho * \rho^\prime$. Consequently $\sigma = f^{-1}(\rho) * f^{-1}(\rho^\prime)$. Let $L= v*v^\prime\subset \sigma$ be any line with $v\in f^{-1}(\rho)$ and $v^\prime \in f^{-1}(\rho^\prime)$. Let $w=f(v), w^\prime = f(v^\prime)$. By linearity of $f$ \[f^{-1}(N_\ep(w))\cap L = N_\delta(v)\cap L\] for $\delta = \dfrac{d(v,v^\prime)}{d(w,w^\prime)}\ep.$ However we have that 
\begin{align}
 2\,\rad(\sigma) \leqslant d(v,v^\prime) \leqslant \diam(\sigma) \notag \\
 2\,\rad(\tau) \leqslant d(w,w^\prime) \leqslant \diam(\tau) \notag
\end{align}
from which the result follows.
\end{proof}

\begin{cor}
Let $X$, $Y$ be l.f. f.d. simplicial complexes equipped with standard metrics. Let $f:X^\prime \to Y^\prime$ be a simplicial map where $X^\prime$ and $Y^\prime$ are subdivisions equipped with the same metrics. Then for all $\rho \in Y^\prime$, \[f^{-1}(N_{k\ep}(\rho)) \subset N_\ep(f^{-1}(\rho)) \subset f^{-1}(N_{K\ep}(\rho)),\] for $k = \dfrac{\mesh(Y^\prime)}{2\,\comesh(X^\prime)}$, $K= \dfrac{2\,\comesh(Y^\prime)}{\mesh(X^\prime)}$.
\end{cor}

\section{Proof}\label{section2}
In this section we prove the main theorem which we restate for convenience:
\begin{thm}[Squeezing]\label{topsqueezing}
 Let $X,Y$ be f.d. l.f. simplicial complexes, then there exists an $\ep=\ep(X,Y)$ such that for any simplicial map $f:X\to Y$, if $f$ is an $\ep$-controlled homotopy equivalence, then $f$ is homotopic to a $Y$-triangular homotopy equivalence.
\end{thm}

By definition an $\ep$-controlled homotopy equivalence spreads things around by at most $\ep$. The idea in proving the thoerem is that for $\ep$ small we may construct retracting maps that map $\ep$-neighbourhoods of each simplex $\sigma$ in an iterated barycentric subdivision back onto $\sigma$. Post-composing with these retracting maps compensates for the $\ep$-control and gives us the $Y$-triangular condition. Explicitly we require the following technical proposition:

\begin{prop}\label{technicalprop}
Let $X$ be a f.d. l.f. simplicial complex, then there exists an $\ep(X)>0$ such that for all $\ep<\ep(X)$ there is an integer $i(X,\ep)$ such that for all integers $i\geqslant i(X,\ep)$ there exists a simplicial approximation to the identity $r: Sd^i\, X \to X$ and a homotopy $P:\id\sim r$ such that for all $\sigma\in X$ and all $0\leqslant\ep^\prime \leqslant \ep$, 
\begin{eqnarray*}
 r(N_\ep(Sd^i\, \sigma))&\subset& \sigma, \\ 
 P(N_{\ep^\prime}(Sd^i\, \sigma),I)&\subset& N_{\ep^\prime}(Sd^i\, \sigma).
\end{eqnarray*}
\end{prop}

\begin{proof}
Set $\ep(X) = \comesh(X)>0$ and take any $\ep< \ep(X)$. Let $i(X,\ep)$ be the smallest integer such that $\mesh(Sd^{i(X,\ep)}\,X) < \ep(X) - \ep$. Note we can find such an integer since dim$(X)<\infty$ and \[\mesh(Sd^j\, X) \leqslant \left(\dfrac{\mathrm{dim}(X)}{\mathrm{dim}(X)+1}\right)^j\mesh(X).\] Hence for all $i\geqslant i(X,\ep)$, $\mesh(Sd^i\,X)< \ep(X)-\ep$ so in particular 
\begin{equation}
D(\widehat{\sigma},Sd^{i-1}\,X) \subset Sd^i\,\sigma \backslash N_\ep(\partial\sigma) \label{inmiddle} 
\end{equation}
for all $\sigma\in X$.

Construct a simplicial approximation to the identity $r:Sd^i\, X\to X$ as the composition $r=r_1\circ\ldots \circ r_i$ of simplicial approximations to the identity $r_j:Sd^j\,X \to Sd^{j-1}\,X$. Let $r_1:Sd\, X \to X$ be any simplicial approximation to the identity and for $j\geqslant 2$ define $r_j:Sd^j\,X\to Sd^{j-1}\,X$ as follows. For all $\tau\in Sd^{j-1}\,X$ there is a unique simplex $\rho\in Sd\,X$ such that $\mathring{\tau}\subset \mathring{\rho}$. Writing $\rho = \widehat{\sigma}_0\ldots \widehat{\sigma}_n$ observe that $\mathring{\tau}\subset \mathring{\sigma}_n$. If $n=0$, then $\widehat{\tau}=\tau = \rho$ is a vertex so we must define $r_j(\widehat{\tau}):= \widehat{\tau}$ as we choose simplicial approximations to send $Sd\, \sigma$ to $\sigma$. Otherwise define $r_j(\widehat{\tau})$ to be any vertex $v$ of $\tau$ that minimises the distance $d_X(v,\widehat{\sigma}_0\ldots\widehat{\sigma}_{n-1})$.

Since $X$ is given the standard metric, if $\sigma^\prime<\sigma$ is a codimension $1$ face, then \[d_X(x,\sigma^\prime) = d_X(x,\partial\sigma),\quad \forall x\in \sigma^\prime*\widehat{\sigma}.\] Thus, for $j\geqslant 2$ we have defined $r_j$ on the vertices of $Sd^j\,X$ to minimise the distance to $\partial\sigma$. As all simplices $\widetilde{\tau}\in Sd^j\,\sigma$ are contained in $\widehat{\sigma}*\sigma^\prime$ for some $\sigma^\prime<\sigma$, all vertices of $\widetilde{\tau}$ are sent towards $\sigma^\prime$ and hence in the same direction towards the boundary.\footnote{Here towards can mean a distance $0$ towards, the point is that things get no further away from the boundary.} Thus by convexity all points in $\widetilde{\tau}$ are mapped towards the boundary by $r_j$, $j\geqslant 2$. Whence 
\begin{equation}\label{rjtoboundary}
d_X(r_j(x),\partial\sigma) \leqslant d_X(x,\partial\sigma),\quad \forall x\in Sd^j\,\sigma,\, j\geqslant 2.
\end{equation}
Define $P_j:\id_X\simeq r_j$ to be the straight line homotopy for all $j$. By equation $(\ref{rjtoboundary})$, if $j\geqslant 2$ then  
\[d_X(P_j(x,t),\partial\sigma) \leqslant d_X(P_j(x,s),\partial\sigma),\quad \forall 0\leqslant s \leqslant t \leqslant 1,\;\forall x\in Sd^j\,\sigma.\] If $j=1$ this condition trivially holds for all $x\in Sd\,(\partial\sigma)$ but need not hold elsewhere.

Let $P:\id_X \simeq r$ be the concatenation \[P= P_1(r_2\circ\ldots r_j)*\ldots *P_{j-1}(r_j)*P_j.\] By the above considerations it is clear that 
\begin{align}
 r(Sd^j\,\sigma \backslash \mathring{D}(\widehat{\sigma},Sd^{j-1}\,\sigma)) \subset \partial\sigma,\quad &\forall \sigma\in X, \notag \\
d_X(P(x,t),\partial\sigma) \leqslant d_X(P(x,s),\partial\sigma),\quad &\forall 0\leqslant s \leqslant t \leqslant 1,\;\forall x\in Sd^j\,\sigma \backslash \mathring{D}(\widehat{\sigma},Sd^{j-1}\,\sigma). \notag
\end{align}
The result now follows from $(\ref{inmiddle})$ as for all $\sigma \in X$, \[ N_\ep(\partial \sigma) \subset Sd^i\,\sigma\backslash D(\widehat{\sigma},Sd^{j-1}\,\sigma).\]
\end{proof}

\begin{ex}
Let $\sigma = v_0v_1v_2$ be a $2$-simplex. Defining a simplicial approximation to the identity $r:Sd^2\,\sigma \to \sigma$ following the procedure in Proposition \ref{technicalprop} we obtain $r$ as depicted in \Figref{Fig:retract}.
\begin{figure}[h!]
\begin{center}
{
  \psfrag{v0}[l]{$v_0$}
  \psfrag{v1}{$v_1$}
  \psfrag{v2}[]{$v_2$}
\includegraphics[width=7cm]{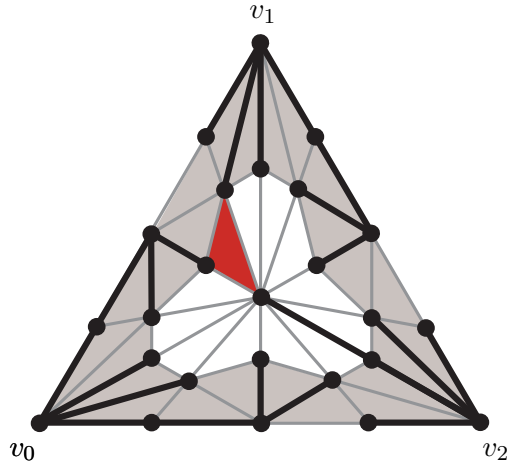}
}
\caption{Constructing $r=r_1\circ r_2$ for the $2$-simplex.}
\label{Fig:retract}
\end{center}
\end{figure}

The region shaded in grey is retracted to $\partial\sigma$ regardless of the choice of where $r_1$ sends $\widehat{\sigma}$. The only $2$-simplex not sent to the boundary is $\Gamma_\sigma(\sigma)$ which is shaded in red - this is always to be found in $D(\widehat{\sigma},Sd\,\sigma)$. The thicker lines are to signify where each vertex is sent. For example, the thicker line between $\widehat{v_0v_2}$ and $\widehat{\widehat{v_0v_2}\widehat{\sigma}}$ tells us that $r_2(\widehat{\widehat{v_0v_2}\widehat{\sigma}}) = \widehat{v_0v_2}$ and the thicker line between $v_0$ and $\widehat{v_0v_2}$ tells us that $r_1(\widehat{v_0v_2}) = v_0$.
\end{ex}

With the retracting maps $r$ constructed as in the technical proposition the proof of the main theorem reduces to verification of the $Y$-triangular condition.

\begin{proof}[Proof of Theorem \ref{topsqueezing}]
Let $k:= \dfrac{\mesh(Y)}{2\,\comesh(X)}$, $K:= \dfrac{2\,\comesh(Y)}{\mesh(X)}$ and \[\ep(X,Y):= \mathrm{min}(k\comesh(X), k\comesh(Y)/K).\]
For all $\ep<\ep(X,Y)$, using Proposition \ref{technicalprop} we obtain integers $i_X:= i(X,\ep)$, $i_Y:=i(Y,\ep)$ such that for all $i\geqslant i_{X,Y}:=\mathrm{max}(i_X,i_Y)$ we obtain simplicial approximations to the identity $r_X:Sd^i\, X \to X$, $r_Y:Sd^i\,Y \to Y$ and homotopies $P_X:\id_X\simeq r_X$, $P_Y:\id_Y\simeq r_Y$ such that 
\begin{eqnarray}
 r_X(N_{\ep/k}(Sd^i\, \sigma))&\subset& \sigma, \quad \forall \sigma \in X, \label{eqnrX} \\ 
 P_X(N_{\ep^\prime}(Sd^i\, \sigma),I)&\subset& N_{\ep^\prime}(Sd^i\, \sigma), \quad \forall \ep^\prime \in [0, \ep/k],\; \sigma \in X, \label{eqnPX} \\
 r_Y(N_{K\ep/k}(Sd^i\, \tau))&\subset& \tau, \quad \forall \tau \in Y, \label{eqnrY} \\ 
 P_Y(N_{\ep^\prime}(Sd^i\, \tau),I)&\subset& N_{\ep^\prime}(Sd^i\, \tau), \quad \forall \ep^\prime \in [0, K\ep/k]\; \tau \in Y. \label{eqnPY} 
\end{eqnarray}

Let $f:X\to Y$ be a simplicial $\ep$-controlled homotopy equivalence with inverse $g$ and homotopies $h_1:fg\sim \id_Y$, $h_2:gf\sim \id_X$. Then $f^\prime:= r_Y\circ f$ is a $Y$-triangular homotopy equivalence with inverse $g^\prime:= r_X\circ g$. Note first that any map $f^\prime:(X,f^\prime) \to (Y,\id_Y)$ is automatically $Y$-triangular with respect to itself, so the first thing we need to check is whether $g^\prime$ is a $Y$-triangular map with respect to $f^\prime$:
\begin{eqnarray*}
g^\prime(\mathring{\sigma})= r_X(g(\mathring{\sigma})) &\subset & r_X(f^{-1}(N_\ep(\sigma))), \quad\mathrm{since}\;g\;\mathrm{is}\;\ep\mathrm{-controlled}, \\
&\subset& r_X(N_{\ep/k}(f^{-1}(\sigma))),\quad\mathrm{by}\; \mathrm{Lemma} \; \mathrm{\ref{annoyinglemma}, } \\
&\subset& f^{-1}(\sigma),\quad\mathrm{by}\; \mathrm{equation} \; \mathrm{(\ref{eqnrX}).}
\end{eqnarray*}
Since $f\sim f^\prime$ and $g\sim g^\prime$ it is automatic that $g$ is a homotopy inverse to $f$.

Now we check that the new homotopies are also $Y$-triangular:
\begin{eqnarray*}
f^\prime\circ g^\prime &=& r_Y\circ f\circ r_X \circ g \\
&\simeq& r_Y\circ f \circ g \\
&\simeq& r_Y \\
&\simeq& \id_Y.
\end{eqnarray*}
Since the composition of $Y$-triangular homotopies is again $Y$-triangular it suffices to check each of the above homotopies is $Y$-triangular. 

The first is $r_Y\circ f\circ P_X \circ g$. We verify that it is $Y$-triangular using an arbitrary $\mathring{\sigma}\in Y$:
\begin{eqnarray*}
 r_Y(f(P_X(g(\mathring{\sigma}),I))) &\subset& r_Y(f(P_X(f^{-1}(N_\ep(\sigma)),I))),\quad\mathrm{since}\;g\;\mathrm{is}\;\ep\mathrm{-controlled}, \\
 &\subset& r_Y(f(P_X(N_{\ep/k}(f^{-1}(\sigma)),I))),\quad\mathrm{by}\; \mathrm{Lemma} \; \mathrm{\ref{annoyinglemma}, } \\
 &\subset& r_Y(f(N_{\ep/k}(f^{-1}(\sigma)))),\quad\mathrm{by}\; \mathrm{equation} \; \mathrm{(\ref{eqnPX}),} \\
 &\subset& r_Y(f(f^{-1}(N_{K\ep/k}(\sigma)))),\quad\mathrm{by}\; \mathrm{Lemma} \; \mathrm{\ref{annoyinglemma}, } \\
 &\subset& r_Y(N_{K\ep/k}(\sigma)), \\
 &\subset& \sigma, \quad\mathrm{by}\; \mathrm{equation} \; \mathrm{(\ref{eqnrY}).}
\end{eqnarray*}

The second homotopy, $r_Y\circ h_1$, is seen to be $Y$-triangular by the following calculation:
\begin{eqnarray*}
 r_Y(h_1(\mathring{\sigma},I)) &\subset& r_Y(N_\ep(\sigma)),\quad\mathrm{since}\;h_1\;\mathrm{is}\;\ep\mathrm{-controlled}, \\
 &\subset& \sigma, \quad\mathrm{by}\; \mathrm{equation} \; \mathrm{(\ref{eqnrY}).}
\end{eqnarray*}

The final homotopy $P_Y$ is $Y$-triangular because $P_Y(\mathring{\sigma},I) \subset \sigma$ by equation $(\ref{eqnPY})$. A similar analysis shows that the composition $g^\prime\circ f^\prime$ is $Y$-triangular homotopic to $\id_X$ which completes the proof.
\end{proof}

\bibliographystyle{hep}
\bibliography{spirosbib}{}
\end{document}